\documentclass[a4paper]{amsart}

\pdfoutput=1

\usepackage[T1]{fontenc}
\usepackage[utf8]{inputenc}
\usepackage{microtype}
\usepackage{amsmath,amssymb,braket,tikz}
\usetikzlibrary{tikzmark,calc}
\usepackage{comment}
\usepackage{url}
\usepackage{hyperref}
\usepackage{tikz}
\usepackage{tikz-cd} 
\usepackage{array}
\usepackage[english]{babel}
\usepackage{amsthm}
\usepackage{amsmath}
\usepackage{ dsfont }
\usepackage{caption}
\usepackage{amsfonts}
\usepackage{epigraph}
\usepackage{amssymb}
\usepackage{blindtext}
\usepackage{subcaption}
\usepackage{caption}
\usepackage{tabu}

\usepackage{multicol, latexsym}
\usepackage{blindtext}
\usepackage{mdframed}
\usepackage{ bbold }
\usepackage{afterpage}
\usepackage{comment}
\usepackage{todonotes}
\definecolor{cSofia}{rgb}{0.1,0.45,0.03}
\definecolor{cChristian}{rgb}{0.6,0.5,0.0}

\usetikzlibrary{positioning,shadows,arrows}

\usepackage{tikz}
\usetikzlibrary{trees}
\usepackage{float}
\newcommand{\R}{\mathds{R}}
\newcommand{\Q}{\mathds{Q}}
\newcommand{\Z}{\mathds{Z}}

\newcommand{\N}{\mathds{N}}

\newcommand{\cd}{\text{cd}}
\DeclareMathOperator{\core}{core}
\newcommand{\finecore}{\core^F\!}
\newcommand{\fineCD}[1]{\mu^F\!(#1)}
\newcommand{\fineSpec}[1]{\mathcal{S}^F\!(#1)}
\newcommand{\vac}{\varnothing}

\DeclareMathOperator{\relint}{relint}
\DeclareMathOperator{\pyrOperator}{Pyr}
\newcommand{\pyr}[1][P]{\pyrOperator(#1)}

\newcommand{\fineadjoint}[2][s]{#2^{F(#1)}}

\newcommand{\numerators}[1][d]{\mathcal{I}_#1}

\usepackage{amssymb, amsfonts}
\usepackage{mathtools}

\usepackage{paralist}

\usepackage{xcolor}

\usepackage[english]{babel} 
\usepackage[autostyle]{csquotes} 
\usepackage{tikz, tikz-cd}

\theoremstyle{plain}
\newtheorem{theorem}{Theorem}[section]
\newtheorem*{theorem*}{Theorem}
\newtheorem{prop}[theorem]{Proposition}
\newtheorem{corollary}[theorem]{Corollary}
\newtheorem{lemma}[theorem]{Lemma}

\theoremstyle{definition}
\newtheorem{definition}[theorem]{Definition}
\newtheorem{example}[theorem]{Example}
\newtheorem{remark}[theorem]{Remark}
\newtheorem{question}[theorem]{Question}

\newcommand{\rleft}{\mathopen{}\mathclose\bgroup\left}
\newcommand{\rright}{\aftergroup\egroup\right}

\DeclareMathOperator{\conv}{conv}

\newcommand{\st}{\: | \:}

\newcommand{\<}{\langle}
\renewcommand{\>}{\rangle}

\title{Classifying the Fine Polyhedral Spectrum}

\author{Sofía Garzón Mora}
\address{Sofía Garzón Mora, Mathematik, Freie Universität Berlin, 14195 Berlin, Germany.}
\email{sofia.garzon.mora@fu-berlin.de}

\author{Christian Haase}
\address{Christian Haase, Mathematik, Freie Universität Berlin, 14195 Berlin, Germany.}
\curraddr{}
\email{haase@math.fu-berlin.de}


\begin{document}
\selectlanguage{english}

\begin{abstract}
In this paper, we examine an analogue of the recently solved spectrum conjecture by Fujita in the setting of Fine polyhedral adjunction theory. 
We present computational results for lower-dimensional polytopes, which lead to a complete classification of the highest numbers of this Fine spectrum in any dimension. Moreover, we present a full classification of the Fine spectrum in dimensions one, two and three, while providing a framework for general classification results in any dimension.
\end{abstract}

\maketitle{}

\section{Introduction}
\label{sec:intro}


The study of adjoint linear systems $|K_X + tL|$ on an algebraic
variety $X$, where $K_X$ is the canonical bundle, $t \geq 0$ and $L$
is an ample line bundle, is, in some sense, a precursor of the Minimal
Model Program. When studying adjunction theory for polarized
varieties, Fujita conjectured in \cite{fujita1992kodaira} and
\cite{fujita1995kodaira} that the set of pseudo-effective thresholds
of smooth polarized $d$-folds should be finite away from 0. This
conjecture has recently been proven, see \cite{di17fujita}. Polyhedral
adjunction theory has been studied in \cite{DiR14}, and Fujita's
spectrum conjecture for toric varieties, allowing canonical
singularities, was proven by Andreas Paffenholz in \cite{Paf15}.

The theory of Fine Polyhedral Adjunction, presented in
\cite{mora2023fine}, has been constructed as a variant of polyhedral
adjunction theory \cite{DiR14}, where it has been proven that the
Spectrum Conjecture also holds in the Fine Adjunction case, without an
assumption on the (toric) singularities. Our main theorem in this
paper is a characterization of the highest numbers in the
spectrum. Here, we completely characterize the spectrum in dimensions
one and two and provide some general classification results for
polytopes of any dimension $d$.

We have computed the Fine $\Q$-codegree for most low-dimensional
polytopes classified by volume that have been previously obtained by
Gabriele Balletti \cite{Bal18}. Additionally, we investigate further
classification results computationally by expressing our
classification problem in the language of a mixed integer linear
program (MILP). As a result, we develop further tools aimed at a full
classification of the Fine $\Q$-codegree spectrum, where we see that
this classification not only refines and combines existing results in
Polyhedral Adjunction Theory with the classification of lattice
polytopes, but also reveals the strong structural properties of our
underlying theory.

\section*{Acknowledgments}
\noindent
SGM was supported by the Deutsche Forschungsgemeinschaft
(DFG) under the Graduiertenkolleg ``Facets of Complexity'' (GRK 2434)
and under Germany's Excellence Strategy – The Berlin Mathematics
Research Center MATH+ (EXC-2046/1, project ID: 390685689). CH was supported by the SPP Priority Program 2458: Combinatorial
Synergies. We would like to thank Benjamin Nill and Andreas Paffenholz
for very enlightening discussions and for aiding with computational
questions on the subject.

\section{Preliminaries}
\label{sec:prelims}
\noindent
We start by including here some crucial definitions with minimal
descriptions. For further details, we refer the reader to
\cite{mora2023fine}.

Let $P \subseteq \R^d$ be a $d$-dimensional rational polytope,
described in a unique minimal way as $P = \{x \in \R^d \st \langle
a_i, x \rangle \geq b_i, i=1,...,m\},$ where the $a_i \in (\Z^d)^*$
are the primitive rows of a matrix $A$, meaning they are not the
multiple of another lattice vector, and $b=(b_1,...,b_m) \in \Q^m$. We
refer to $P$ as a \textit{lattice polytope} if its vertices lie in
$\Z^d$.

\begin{definition}
Let $f$ be the affine functional $f(x) = \langle a , x \rangle -b$ for
some $b \in \Q$ and $a$ an element of the dual lattice
$(\Z^d)^*$. Such a functional is said to be \textit{valid} for a
polytope $P$ if for the halfspace $\mathcal{H}_+ := \{x \in \R^d \st
f(x) \geq 0\},$ we have that $P \subseteq \mathcal{H}_+$.
If, in addition, there is some $p \in P$ with $f(p)=0$, i.e., at least
one point of $P$ lies in $\mathcal{H}$, the hyperplane defined by $f$,
we say $f$ is a \textit{tight} valid inequality for $P$.
\end{definition}


\begin{definition} \label{def: fineadj}
Let $\alpha \in (\R^d)^*$. We define the \textit{support} function
associated with $P$ as
$$h_P : (\R^d)^* \to \R, \quad \alpha \mapsto \min_{x \in P} \langle
\alpha, x \rangle.$$ 
For some fixed real number $s > 0$, we may define the \textit{Fine
  adjoint polytope}, which is a rational polytope, as
$$\fineadjoint{P} := \{ x \in \R^d \st \langle a , x \rangle \geq
h_P(a) + s, \text{ for all } a \in (\Z^d)^* \setminus \{0\} \}.$$
\end{definition}

Here we introduce the notion of a relevant inequality.

\begin{definition}
Let $\mathcal{F}$ be the set of all valid inequalities for $P$, where
an element $f \in \mathcal{F}$ is of the form $\langle a_f , x \rangle
\geq b_f$. A valid inequality $f \in \mathcal{F}$ is said to be
\textit{relevant} for $P$ if for some $s > 0$, it holds that
$$\{x \in \R^d \st \langle a_f , x \rangle \geq h_P (a_f) + s \text{ }
\forall f \in \mathcal{F} \} \neq \{ x \in \R^d \st \langle a_f, x
\rangle \geq h_P (a_f) + s \text{ }\forall f \in \mathcal{F} \setminus
\{f\}\}.$$
The valid inequality $f$ is said to be \textit{irrelevant} if it is
not relevant.
\end{definition}

\begin{theorem} \cite[Corollary 2.7]{mora2023fine} \label{cor: relevantineqs}
Let $P$ be a rational polytope of dimension $d$. The valid relevant
inequalities for $P$ of the form $\langle a, \cdot \rangle \geq
h^F(a)$ are among the $a \in (\mathbb{Z}^d)^*$ such that $a \in
\operatorname{conv}( a_1,...,a_m)$ where the $a_i$ for $1 \leq i \leq
m$ are the primitive inward pointing facet normals of $P$.
\end{theorem}

In what follows we will refer to $P$ as the polytope defined in the
following way:
\begin{equation} \label{eq: minimaldescription}
    P = \{x \in \R^d \st \langle a_i, x \rangle \geq b_i, i=1,...,m\},
\end{equation}
where $b_i \in \Q$ and $a_i \in (\Z^d)^*$ are the primitive rows of a
matrix $A$ including all relevant valid inequalities for $P$.

\begin{definition}
The \textit{Fine $\Q$-codegree} of a rational polytope $P$ is 
$$\mu^{F}(P) := (\sup \{ s > 0 \st P^{F(s)} \neq \emptyset \} )^{-1},$$
and the \textit{Fine core} of $P$ is 
$$\finecore(P) := P^{(1/\mu^{F}(P))}.$$
We will refer to the quantity 
$$n^ F\!(P) := \frac{1}{\mu^ F(P)}$$
as the \textit{Fine Number} of $P$.
\end{definition}

From the definition we see for $s \in \Q_{>0}$ 
\begin{gather} \label{eq:homgeneity}
  n^F\!(sP) = s n^F\!(P) \text{ and therefore } s\fineCD{sP} =
  \fineCD{P} \, \text{ and}
  \\
  \text{if } P \subseteq Q, \text{ then } P^{F(s)} \subseteq Q^{F(s)} \text{ so that }
  n^F(P) \leq n^F(Q) \,. \label{eq:monotonicity}
\end{gather}
The monotonicity \eqref{eq:monotonicity} does not hold for non-Fine adjunction.
We now introduce one more definition, adapted from \cite{DiR14} which
will be very useful for our computations in later sections.

\begin{definition}
Let $P$ be a $d$-dimensional polytope. The Fine Mountain $M^F(P)
\subseteq \R^{d+1}$ of $P$ is given by 
$$M^F(P) := \{(x,s) \st x \in P, 0 \leq s \leq \langle a,x\rangle -
h_P(a), \text{ for all } a \in (\Z^d)^* \setminus \{0\} \}$$
and if $P$ has its inequality description as in \eqref{eq:
  minimaldescription}, then $M^F(P)$ can be described by 
$$M^F(P) = \{(x,s) \in \R^{d+1} \st (A | - \mathbb{1})(x,s)^T \geq b,
s \geq 0 \}.$$ 
Therefore, the Fine mountain of $P$ is a polytope such that for all
$s_0 \geq 0$ 
$$M^F(P) \cap \R^d \times \{s_0\} = P^{F(s_0)} \times \{s_0\}.$$
\end{definition}

As a means of comparison, we will now define the notion of codegree
which comes up in the context of Ehrhart Theory of lattice polytopes
\cite{BN07}.

\begin{definition}
Let $P$ be a lattice $d$-polytope. We define the \textit{codegree} as
\[
\operatorname{cd}(P) := \min \{k \in \N_{\geq 1} \st
\operatorname{int}(kP) \cap \Z^d \neq \emptyset \}.
\]
Then the \textit{degree} of $P$ is $\deg(P) \coloneq d+1 -
\operatorname{cd}(P)$.
\end{definition}
Observe that $\emptyset \neq \operatorname{int}(\operatorname{cd}(P)
P) \cap \Z^d \subseteq \fineadjoint[1]{(\operatorname{cd}(P) P)}$ so
that $\fineCD{P} \leq \operatorname{cd}(P)$.

\subsection{Finiteness of the Fine $\Q$-codegree spectrum}
\label{sec: finespectrum}

It is proven in \cite{mora2023fine} that in each dimension, when
bounded from below by some $\varepsilon >0$, the set of values that
the Fine $\Q$-codegree can take is finite.

In \cite{Paf15}, in particular, the following Theorem is shown.

\begin{theorem}[Paffenholz, {\cite[Theorem 3.1]{Paf15}}] \label{Thm: finiteqcodspec}
Let $d \in \N$ and $\varepsilon > 0$ be given. Then
$$\{\mu(P) \st P  \text{ is a $d$-dimensional lattice polytope with
  canonical normal fan}\} \cap \R_{\geq \varepsilon}$$
is finite.
\end{theorem}

Note that in the context of algebraic geometry, this is Fujita's
conjecture for the case of toric varieties with at most canonical
singularities.

The following Fine version of this result has been proven in
\cite{mora2023fine}. It allows us to ask the classification question
of the Fine spectrum
$$\fineSpec{d} \coloneq \left\{\fineCD{P} \st P
  \text{ is a $d$-dimensional lattice polytope} \right\}$$
that this paper is about.

\begin{theorem}\cite[Theorem 5.4]{mora2023fine} \label{thm:finitespectrum}
Let $d \in \N$ and $\varepsilon > 0$ be given. Then
$ \fineSpec{d} \cap \R_{\geq \varepsilon}$ is finite.
\end{theorem}

For details of this proof, we refer the reader to the original
paper. However, it is worth mentioning here that the proof consists of
two main parts. For the first part, and for future use, we need the
notion of Fine core normal:

\begin{definition}
  A functional $a \in (\Z^d)^*$ is a \textit{Fine core normal} if
  $$\langle a_0 , y \rangle = h_P(a_0) + \fineCD{P}^{-1}$$
  for all $y \in \operatorname{core}^F(P)$.
  The convex hull in $(\R^d)^*$ of the Fine core normals will be
  denoted $A_{\core}^F$.
\end{definition}

It was shown that up to lattice equivalence, for a fixed $d \in
\Z_{>0}$, there are only finitely many sets which can be the Fine core
normals for $d$-dimensional lattice polytopes. Then it was shown that
each such configuration of Fine core normals gives rise to finitely
many values for the Fine $\Q$-codegree above any positive threshold
$\varepsilon$.

If we let $P$ be described by all relevant inequalities as in
\eqref{eq: minimaldescription}, up to relabeling, we can
assume that the set of Fine core normals consists of $a_1,...,a_k$ for
some $k \leq m$. This yields a set of inequalities defining the
affine hull of the Fine core of $P$,
$\operatorname{aff}(\finecore(P))$, that is,

\begin{align*}
  \operatorname{aff}\left( \finecore(P) \right)
  & = \left\{ x \st \langle a_i , x \rangle \ge b_i + \fineCD{P}^{-1}, 1
    \leq i \leq k \right\} \\
  & = \left\{ x \st \langle a_i , x \rangle = b_i + \fineCD{P}^{-1}, 1
    \leq i \leq k \right\} .
\end{align*}

We end this section with the remark below.

\begin{remark} \label{rmk: muofsublattice}
Let $\Lambda \subseteq \Lambda'$ be a sublattice of the lattice
$\Lambda '$ and let $P$ be a rational polytope with respect to the
lattice $\Lambda$, then if we denote by $\fineCD{P,\Lambda}$ the Fine
$\Q$-codegree of $P$ with respect to the lattice $\Lambda$, we have
that
\[
\fineCD{P,\Lambda} \geq \fineCD{P, \Lambda'}.
\]
\end{remark}

\section{Tools}
\label{sec: tools}

Before we explore the Fine spectrum in small dimensions and to some
extent in general dimension in \S\ref{sec: classification}, we develop
some general tools to facilitate the search.

\subsection{Project and Lift}
\label{sec:project-lift}
Now, we want to study the behaviour of the Fine $\Q$-codegree under
lattice preserving projections $\pi \colon \R^d \to \R^{d'}$.

\begin{remark} \cite[Remark 3.4]{mora2023fine} \label{rmk: anyprojection}
  Under any projection $\pi$ for any rational polytope $P \subset
  \R^d$ we have $\fineCD{\pi(P)} \leq \fineCD{P}$.
\end{remark}

But every $P$ comes with its own preferred projection.

\begin{definition}
Let $K(P)$ be the linear space parallel to
$\text{aff}(\finecore(P))$. Then the projection $\pi_P : \R^n \to
\R^n/K(P)$ is called the natural projection associated with $P$.
\end{definition}

It has the following nice properties.

\begin{prop}\cite[Proposition 3.3]{mora2023fine} \label{Prop:coreispt}
The image $Q := \pi_P(P)$ of the natural projection of $P$ is a
rational polytope satisfying $\fineCD{Q} = \fineCD{P}$. Moreover
$\finecore(Q)$ is the point $\pi_P(\finecore(P))$.
\end{prop}

Using the natural projection, we can relate the Fine spectra in
different dimensions.

\begin{lemma}
  For $d \in \Z_{>1}$, 
  $\fineSpec{d-1} \subset \fineSpec{d}$, and
  $$
  \fineSpec{d} \setminus \fineSpec{d-1} \subset \left\{ \fineCD{P} \st
    P \text{ a lattice $d$-polytope with } \finecore(P)
    \text{ a point} \right\}\,.
  $$
\end{lemma}

\begin{proof}
  For the first inclusion, observe that $\fineCD{P \times [0,h]} =
  \fineCD{P}$ for $h \gg 0$.

  The second inclusion is a restatement of Proposition~\ref{Prop:coreispt}.
\end{proof}

\subsection{Computational approaches}
\label{sec:computation}
As we began this project, we started by computing the Fine
$\Q$-codegrees of all lattice $d$-polytopes for $d \le 5$ of small volume as classified
by Balletti~\cite{Bal18}. We summarize our findings below.

Then, in \S\ref{sec:MILPs}, we present a family of mixed integer
linear programs to explore the Fine spectrum.

\subsubsection{Results from polytope database}
Gabriele Balletti compiled a database of all lattice polytopes, up 
to isomorphism, of dimension at most five and of small\footnote{What 
volume is ``small'' depends on the dimension as in the second column 
of Table~\ref{fig:tablespec}.} volume~\cite{Bal18}. 

For each of these polytopes, we computed the Fine $\Q-$codegree.
The main idea of our algorithm is that, given $d \in
\Z_{>0}$ and a polytope $P$ of dimension $d$, we obtain the relevant
valid inequalities for $P$ via Corollary \ref{cor: relevantineqs} and
in this way we construct the Fine mountain for $P$. The Fine
$\Q$-codegree can be obtained then, as being the height of this Fine
mountain. Moreover, some other data from the polytopes can be obtained
via this algorithm, such as the dimension of the Fine core of $P$
which corresponds to the dimension of the topmost face of the Fine
mountain of $P$ or the Fine core vertices, corresponding to the
vertices of this topmost face of the Fine mountain.

The code for computing the Fine mountain, the Fine $\Q$-codegree, and
the dimension of the Fine core has been implemented using
\texttt{Polymake}\cite{gawrilow2000polymake, assarf2017computing}
software for polytopes. Following these
computations, we have observed repeating behaviors which emerge in
various dimensions for the Fine $\Q-$codegree.

\begin{table}[H]
\begin{center}
\begin{tabular}{ |c|c|c| } 
 \hline
 Dimension & Max. Volume & Fine $\Q$-codegree spectrum \\
 \hline
 &&\\[-0.5em]
 1 & $\infty$ & $2, 1, \frac{2}{3}, \frac{1}{2}, \frac{2}{5},
                \frac{1}{3}, \frac{2}{7}, \frac{1}{4},...$  \\[2ex] 
 \hline
 &&\\[-0.5em]
 2 & 50 & $3, 2, \frac{3}{2}, 1, \frac{3}{4}, \frac{2}{3},
          \frac{3}{5}, \frac{1}{2}, \frac{3}{7}, \frac{2}{5},
          \frac{3}{8}, \frac{1}{3}, \frac{3}{10}, \frac{2}{7}$
  \\[2ex] 
 \hline
 &&\\[-0.5em]
 3 & 23 & $4, 3, \frac{5}{2}, 2, \frac{3}{2}, \frac{4}{3},
          \frac{5}{4}, \frac{7}{6}, 1, \frac{5}{6}, \frac{4}{5},
          \frac{3}{4}$ \\ [2ex]
 \hline
 &&\\[-0.5em]
 4 & 14 & $5, 4, \frac{7}{2}, 3, \frac{5}{2}, \frac{7}{3},
          \frac{9}{4}, 2, \frac{7}{4}, \frac{5}{3}, \frac{3}{2},
          \frac{4}{3}, 1$ \\[2ex]
 \hline
 &&\\[-0.5em]
 5 & 14 & $6, 5, \frac{9}{2}, 4, \frac{7}{2}, \frac{10}{3},
          \frac{13}{4}, 3, \frac{11}{4}, \frac{8}{3}, \frac{5}{2},
          \frac{7}{3}, \frac{9}{4}, 2, \frac{7}{4}, \frac{5}{3},
          \frac{3}{2}, 1$\\[2ex]
 \hline
\end{tabular}
\end{center} 
    \caption{Fine $\Q$-codegree spectra}
    \label{fig:tablespec}
\end{table}

One of the first observations from this table, which also holds in higher
dimensions (cf.~Theorem~\ref{thm: general d}) is that the Fine $\Q$-codegree of a $d$-dimensional polytope $P$ is at most $d+1$ with equality only for the
standard $d$-dimensional simplex, and that $\fineCD{P}$ will not take any values between $d$ and $d+1$.

\subsubsection{A family of MILPs}
\label{sec:MILPs}
Recalling the proof of Theorem \ref{thm:finitespectrum} we deduce a way of exploring the Fine spectrum. First of all, for a fixed dimension $d$, we need to determine all possible different configurations of Fine core normals. Note that these correspond to lattice polytopes whose only interior lattice point is the origin, for which there are only finitely many possibilities \cite[Theorem 1]{LZ91}. We then fix a set of Fine core normals and proceed to find all possible polytopes with such a Fine core normal configuration that give rise to different $\mu^F$ values. As we shall see, this corresponds to solving a Mixed Integer Linear Program.

It is worth remarking that, in what follows, we can consider only polytopes which have a zero-dimensional Fine core using the arguments from \S\ref{sec:project-lift}.

Fix a dimension $d$ and a configuration $\mathcal{A}$ of Fine core normals $a_1,...,a_n \in (\R^d)^*$. We want to construct a polytope $P$ with zero-dimensional Fine core, namely $\finecore(P) = \{p\}$, with vertices $u_1,...,u_n \in \Z^d$, that satisfies the property of having these Fine core normals. Let us take as variables the $u_i$ for all $1 \leq i \leq n$, $p \in \Q^d$ and $n^F := \frac{1}{\fineCD{P}} \in \Q$ the \textit{Fine number} of $P$. Then the Mixed Integer Linear Program can be stated as: 

\begin{equation} \label{eq: milp} \tag{MILP$^\pm_A(L,U)$}
\begin{array}{ll@{}ll}
\text{maximize} & \pm n^F &\\
\text{subject to}&  \langle a_i, u_j \rangle & \geq  \langle a_i , u_i \rangle & \text{ for all } i \neq j \\
                 &   \langle a_i, p \rangle  & = \langle a_i , u_i \rangle + n^F &  \text{ for all } i \\
                 &   n^F & \in [L,U] & \\
                 & p & \in \Q^d \cap [0,1]^d  \\
                 & u_1,\dots,u_n & \in \Z^d. \\
                 
\end{array}
\end{equation}

Solving this MILP will return vectors $u_1,\dots,u_n \in \Z^d$ whose convex hull defines a polytope with a Fine core normal configuration given by $a_1,\dots,a_n \in (\R^d)^*$ and with $L \leq n^F \leq U$, where $L$ and $U$ are some lower and upper bounds for the Fine number. It is possible to explore the Fine spectrum by maximizing $n^F$ and varying $U$ or maximizing $-n^F$ and varying $L$, thus achieving a complete list of all values for $\mu^F$ by varying these two parameters consistently. It is not immediately clear how small the variations of $L$ and $U$ should be in general to guarantee a complete scanning of the Fine spectrum.


This characterization of the values $\fineCD{P}$ requires a classification of lattice polytopes whose only interior lattice point is the origin. In fact, a classification of inclusion minimal such polytopes is sufficient, see \ref{rmk: muofsublattice}. In the next section, we will study the cases of lower dimensions and introduce some results that hold in a general dimension $d$. However, to make our endeavor more manageable, we will first discuss some interesting results that restrict our spectrum values in any dimension $d$.

\subsection{Reducing the search space}
\label{sec:circuits-pyramids-determinants}


\begin{lemma} \label{lemma: dimdcircuits}
  Let $P$ be a lattice polytope of dimension $d$, and
  suppose there are Fine core normals $a_0,...,a_\ell$ and
  coefficients $\lambda_0, \ldots, \lambda_\ell \in \R_{>0}$ so 
  that $\sum_{i=0}^\ell \lambda_i a_i =0$, then there is a lattice
  polytope $Q$ of dimension $\le \ell$ such that
  \[
    \fineCD{P} = \fineCD{Q}.
  \]
\end{lemma}

\begin{proof}
  For the conclusion we can safely assume $\ell < d$.

  Look at the projection
  $\pi = (a_1, \ldots, a_\ell) \colon \R^d \to \R^\ell$. It takes $P$
  to $Q \coloneq \pi(P)$, considered as a lattice polytope in the
  sublattice $\Lambda \coloneq \pi(\Z^d) \subseteq \Z^\ell$.
  Our claim is that this projection is Fine $\Q$-codegree preserving,
  i.e., $\fineCD{P} = \fineCD{Q}$.

  First, observe that $a_0, \ldots, a_\ell$ descend to well-defined
  elements of $\Lambda^\vee$ as $\ker \pi \subseteq \ker a_i$ for
  $i=0, \ldots, \ell$.

  Next, recall from Remark \ref{rmk: anyprojection} that $\fineCD{\pi(P)} \leq \fineCD{P}$
  for any projection. So we just need to verify that the reverse
  inequality holds.

  There are $v_0,...,v_\ell \in P \cap \Z^d$ and $b_0,...,b_\ell \in
  \Z$ such that for $x \in \relint(\finecore(P))$, we have
  \[
    b_i = h_P(a_i) = \langle a_i, v_i \rangle = \langle a_i , x \rangle -
    \frac{1}{\fineCD{P}}
  \]
  for all $0 \leq i \leq \ell$.
  The relation $\sum_{i=0}^\ell \lambda_i a_i =0$ yields a certificate
  that for every $s > 1/\fineCD{P}$, the system of inequalities
  \[
    \< a_i, y \> \ge b_i + s \quad (i=0, \ldots, \ell)
  \]
  for $y \in \R^d$ is infeasible.
  As mentioned above, these inequalities descend to valid inequalities
  for $Q$. This shows $\fineadjoint{Q} = \vac$ for all $s >
  1/\fineCD{P}$, in other words $\fineCD{Q} \ge \fineCD{P}$.
\end{proof}

Now we show that it is enough to consider Fine core normal configurations given by the vertices of lattice polytopes with the origin as the only interior lattice point, since any possible sub-configuration is already accounted for in this classification.

\begin{lemma} \label{lemma:enoughWithVertices}
Let $v_1,...,v_k$ be Fine core normals of a polytope $P$. Let $v = \sum_{i=1}^k \lambda_i v_i \in \Z^d$ where $\sum_{i=1}^k \lambda_i \leq 1$ and $v \neq 0$. Then $v$ is also a Fine core normal of $P$.
\end{lemma}

\begin{proof}

Let $b_i = \min_{p \in P} \langle v_i , p \rangle$ and let $x \in \operatorname{relint}(\finecore(P))$, then we have by definition $\langle v_i, x \rangle = b_i + n^F(P)$. Denote $b := \min_{p \in P} \langle v,p \rangle$ and since $v = \sum_{i=1}^k \lambda_i v_i$, we have that 
\[
b \geq \sum_{i=1}^k \lambda_i b_i.
\]
Moreover, also by the definition of $n^F(P)$ we have
\[
\langle v, x \rangle \geq b + n^F(P).
\]
In order to show that $v$ is a Fine core normal of $P$, we want to show that the above inequality is indeed an equality. But
\[
n^F(P) \leq \langle v ,x \rangle - b \leq \langle v, x \rangle - \sum_{i=1}^k \lambda_i b_i = \sum_{i=1}^k \lambda_i \left( \langle v_i, x \rangle - b_i \right) = n^F(P) \sum_{i=1}^k \lambda_i \leq n^F(P),
\]
where the first inequality follows from the definition of $n^F(P)$, the second one by the relation between $b$ and the $b_i$, and the last one since by assumption $\sum_{i=1}^k \lambda_i \leq 1$. Thus, we obtain $\langle v, x\rangle = b+n^F(P)$ as desired.
\end{proof}

In order to further reduce our database, we want to show that the set of numerators of the Fine $\Q$-codegree is finite. We first define this set.

\begin{definition} \label{def: IdNumerators}
For a fixed dimension $d$, let $\numerators \subset
\Z_{>0}$ be the set of all possible numerators of the rational number $\fineCD{P}$, where $P$ is a lattice $d$-polytope.
\end{definition}

The theorem below provides us with a way of computing the elements of the set $\numerators$.

\begin{theorem} \label{thm: FiniteNumerators} 
Fix a dimension $d \in \Z_{>0}$. Consider the set $\mathcal{F}_d$ of $d$-dimensional Fine core normal configurations and for any such configuration $\mathcal{A}$, choose a set $a_1,\dots,a_{d+1} \subseteq \mathcal{A}$ of vectors which are positively spanning. To such a choice, we associate a matrix $A$ whose $i$-th row is $a_i$ for $i=1,...,d+1$ and a number 
\[
\eta_{A} := \frac{\det \begin{bmatrix} A & - \mathbb{1} \end{bmatrix}}{D}
\]
where $D := \gcd_{j=1,\dots,d+1}\{ \det \begin{bmatrix} A_j \end{bmatrix} \}$ and $A_j$ is the matrix $A$ with the $j$-th row removed. Then
\[
\numerators = \mathcal{I}_{d-1} \cup \{\eta_{A} \st  a_1,\dots,a_{d+1} \subset \mathcal{A} \text{ positively spanning}, \mathcal{A} \in \mathcal{F}_d \}.
\]

\end{theorem}

\begin{proof}
For all $d > 0$, $\mathcal{I}_d \subseteq \mathcal{I}_{d+1}$ due to the arguments from \S\ref{sec:project-lift}. Let $\finecore(P) = \{p\}$. By definition, the Fine core normals satisfy
\[
\langle a_i , p \rangle = b_i + n^F(P)
\]
for some $b_i \in \Q$, $i=1,\dots,d+1$. Note that we can indeed find $d+1$ linearly independent positively spanning Fine core normals, since the Fine core of $P$ consists of one point. Let $A := A_{a_1,\dots,a_{d+1}}$ be given as above and $b = (b_1,\dots,b_{d+1})$. Then, $p$ and $n^F(P)$ satisfy the following linear system:
\[
\begin{bmatrix}
    A & - \mathbb{1} 
\end{bmatrix} \begin{pmatrix}
    p \\ n^F(P)
\end{pmatrix} = b.
\]
By Cramer's rule, the system has a unique solution. In particular
\begin{equation} \label{eq:FineNumberDet}
   n^F(P) = \frac{\det \begin{bmatrix} A & b \end{bmatrix}}
{\det \begin{bmatrix} A & - \mathbb{1} \end{bmatrix}}, 
\end{equation}
so that $\det \begin{bmatrix} A & - \mathbb{1} \end{bmatrix} \in \mathcal{I}_d$ for each such choice $a_1,\dots,a_{d+1}$ of $d+1$ Fine core normals out of the initial configuration $\mathcal{A}$. In particular, any element in $\numerators$ occurs as such a determinant or one of its divisors. 

Conversely, given such a matrix $A$ we can find a vector $b$ so that the convex hull of the lattice points inside $Ax \leq b$ realizes the Fine number as in \eqref{eq:FineNumberDet} (cf. Section \ref{sec:RealizingNumerator}).


\end{proof}

\subsection{Realizing a Fine $\Q$-codegree numerator}
\label{sec:RealizingNumerator}

Fix a dimension $d \in \Z_{>0}$ and take a given matrix $A$, whose rows are positively spanning vectors $a_1,\dots,a_{d+1} \subseteq \mathcal{A}$ for $\mathcal{A}$ a Fine core normal configuration as in the statement of Theorem \ref{thm: FiniteNumerators}. We study now how to realize the determinant $\det \begin{bmatrix} A & - \mathbb{1} \end{bmatrix}$ in practice as the Fine number $n^F(P)$ of a lattice polytope $P$ as in \eqref{eq:FineNumberDet}, which we can assume contains the origin in its relative interior. For this, we make use of \eqref{eq: milp} and state this as a remark. 

\begin{remark}
Given such a matrix $A$, minimizing $n^F$ in \eqref{eq: milp} with the rows of $A$ as the vectors $a_1,\dots,a_{d+1}$ of a Fine core normal configuration, is equivalent to minimizing $\det \begin{bmatrix} A & b \end{bmatrix}$. By construction, the MILP will output a lattice polytope $P$ whose Fine number realizes $\det \begin{bmatrix} A & - \mathbb{1} \end{bmatrix}$ as the numerator of its Fine $\Q$-codegree, thereby finding an appropriate vector $b$ which describes the desired polytope.
\end{remark}



\subsection{Lattice Pyramids}
\label{sec:latticepyramids}

Now we present some results on lattice pyramids and the behavior of their Fine $\Q$-codegree. Here we need the following definition, which we phrase similarly as in
\cite{latticePyramid}.

\begin{definition} \label{def: LatticePyramid}
Let $P$ be a $d$-dimensional lattice polytope. We say $P$ is a
\textit{lattice pyramid} if $P$ is a lattice point or if there is a
lattice polytope $P' \subset \R^{d-1}$ such that $P \cong
\operatorname{conv}(P' \times \{ 1\}, \{ 0\}) \subset \R^{d-1} \times
\R$. For instance, the standard simplex $\Delta_d$ is a lattice
pyramid.
\end{definition}

Using explicit examples from which we obtained the data presented in
the table, we have observed that there are certain lattice polytopes
$P$ with Fine $\Q$-codegree $\fineCD{P} > 1$ such that the value
$\fineCD{P}+1$ is taken by a lattice pyramid over $P$. Note that if
$\fineCD{P} < 1$, then the lattice pyramid over $P$ will have Fine
$\Q$-codegree $2$. However, it is not straightforward to describe the
Fine $\Q$-codegree of a lattice pyramid, as the following example
illustrates.

\begin{example}
\label{eg:rational-pyramid}
If we consider a lattice pyramid over the rational segments $S_1 =
[-\frac{1}{5}, \frac{2}{5}]$ and $S_2 = [\frac{1}{5}, \frac{4}{5}]$,
we have that $\fineCD{S_1} = \frac{10}{3} = \fineCD{S_2}$ but
$\fineCD{\operatorname{Pyr}(S_1)} = \frac{10}{3}$ and
$\fineCD{\operatorname{Pyr}(S_2)} = 3$.
\end{example}

This observation leads us to a result which will be used below in our main
classification Theorem.

\begin{theorem} \label{thm: MuFofPyramid}
Let $P$ be a $d$-dimensional lattice polytope and $\pyr$ a
lattice pyramid over $P$. Then 
$$\fineCD{\pyr} = \max \{ 2, \fineCD{P}+1 \}.$$
\end{theorem}

\begin{proof}

For convenience denote $\mu := \fineCD{P}$.
We first show the lower bounds.

The valid inequalities $0 \le x_{d+1} \le 1$ for $\pyr$ show that
$\fineadjoint{\pyr} = \vac$ for $s > 1/2$. Hence, $\fineCD{\pyr} \ge 2$.

To prove $\fineCD{\pyr} \ge \mu+1$ we show that, if
$\fineadjoint[t]{\pyr} \neq \vac$, then 
$\fineadjoint[s]{P} \neq \vac$ for $s = \frac{t}{1-t}$.
So suppose $(p,h) \in \fineadjoint[t]{\pyr}$. We claim that $p/h \in
\fineadjoint[s]{P}$. Consider a valid inequality $\langle a, x \rangle
\ge b$ for $P$.
Then $\langle a, x \rangle - b x_{d+1} \ge 0$ is valid
for every $(x,x_{d+1}) \in \pyr$. Then $(p,h) \in
\fineadjoint[t]{\pyr}$ implies $\langle a, p \rangle - bh
\ge t$. We obtain
$
\langle a, p/h \rangle \ge b + t/h \ge b + s
$
where we use $h \le 1-t$ as $(p,h) \in \fineadjoint[t]{\pyr}$.

For the upper bound on $\fineCD{\pyr}$,
let $p \in \operatorname{relint}(\finecore(P))$. We want to show that
$(\hat{p}, \hat{h}) \coloneq (1-s) (p,1) \in \fineadjoint{\pyr}$ for
$s = \min \left\{\frac12, \frac{1}{\mu +1} \right\}$.

Let $\langle a, x \rangle + \alpha x_{d+1} \ge b$ with be some
inequality valid for all $(x,x_{d+1}) \in \pyr$. By considering a
$\Z^d$-translate of $P$ (this corresponds to applying a unimodular
transformation to $\pyr$ which leaves the functional $a$ unchanged),
we may assume that $a$ takes its minimum over $P$ at $0$.  
Then $\< a,p \> \ge 1/\mu \ge \frac{s}{1-s}$.

First, assume that $\alpha \leq 0$. In this case, the functional
$(a,\alpha)$ achieves its minimum along $\pyr$ at the vertex
$e_{d+1}=(0,...,0,1)$, and this minimum is $\alpha$.
Then
$$
\< (a,\alpha) , (\hat p, \hat h) \> = (1-s) \ \< (a,\alpha) , (p, 1) \> 
\ge (1-s) \left(
  \frac{s}{1-s} + \alpha \right) \ge s + \alpha \,.
$$
as desired.

If, on the other hand, $\alpha \ge 1$, the functional
$(a,\alpha)$ achieves its minimum along $\pyr$ at the origin
and this minimum is $0$.
Then 
$$
\langle (a,\alpha), (\hat{p}, \hat{h}) \rangle = (1-s) \langle
(a,\alpha), (p,1) \rangle \ge (1-s) \ \left(
  \frac{s}{1-s} + \alpha \right) \ge s + 0
$$
\end{proof}

\section{Classification results}
\label{sec: classification}

In this section, we want to get to know the Fine spectra
$\fineSpec{d}$. This is easy for lattice segments:
\begin{equation}
  \label{eq:dimension1spectra}
  \fineSpec{1} = \left \{ \frac{2}{k} \text{ } \Bigl\vert \text{ } k \in
    \Z_{>0} \right \} \,,
\end{equation}
but already non-trivial for polygons (cf.~\S\ref{sec:d=2}), where we
obtain a full characterization.
After developing some methods in~\S\S\ref{sec:project-lift} \&
\ref{sec:circuits-pyramids-determinants} and some computational tools
in~\S\ref{sec:computation}, we narrow down $\fineSpec{3}$
(\S~\ref{sec:d=3}), and we determine the finite set $\fineSpec{d} \cap
\Q_{\ge d-1}$ for all $d$ in \S\ref{sec:general-d}.

\subsection{$d=2$}
\label{sec:d=2}
To classify the Fine spectrum in dimension $2$, we need to first describe all possible Fine core normal configurations. These correspond to $2$-dimensional lattice polytopes with the origin as the sole interior lattice point, or in this case, to 16 reflexive polytopes, see \cite{latticepolygons12, kohl2020unconditional}.

After filtering through the 16 configurations using Lemma \ref{lemma: dimdcircuits}, we obtain just one configuration which does not have opposing vertices. We state this as a remark.

\begin{remark} \label{rmk: only2DimConfig}
The only $2$-dimensional Fine core normal configuration which does not
have opposing vertices is given by the vectors $\mathcal{A} = \{ (1,0),
(0,1), (-1,-1) \}$.
\end{remark}

One can alternatively use the classification of toric $l$-reflexive surfaces
for the index $l=1$ by Kasprzyk and Nill in \cite{KN11} as part of the Graded Ring Database \cite{BK} or the classification in \cite{MR546291} where we are only interested in the $2$-dimensional polytope whose Fine core normal configuration corresponds to the normal fan of the projective space $\mathbb{P}^2$ as a toric variety. We now state our classification result.

\begin{theorem} \label{thm: 2DimSpectrum}
The $2$-dimensional Fine $\Q$-codegree spectrum has the form 
\[
S^F(2) = \left \{ \frac{2}{k} \text{ } \Bigl\vert \text{ } k \in
  \Z_{>0} \right\} \cup \left \{ \frac{3}{k} \text{ } \Bigl\vert
  \text{ } k \in \Z_{>0} \right\} .
\]
\end{theorem}

\begin{proof}
By Lemma \ref{lemma: dimdcircuits} and Section \ref{sec:project-lift}, we know that if a reflexive $2$-dimensional lattice polytope $P$ contains opposing vertices, then $\frac{1}{\fineCD{P}} = \frac{k}{2}$ for some $k \in \Z_{>0}$, which is the case for 15 out of the 16 reflexive polytopes. Also, for a given such $k \in \Z_{>0}$, we can find a polytope $P$ with $\frac{1}{\fineCD{P}} = \frac{k}{2}$ by \eqref{eq:dimension1spectra} and the same projecting and lifting arguments.

Hence, the only Fine core normal configuration we are interested in is $\mathcal{A}$ as in Remark \ref{rmk: only2DimConfig}.For
the standard simplex it holds that $\fineCD{\Delta_2} =3$, so
using the homogeneity~\eqref{eq:homgeneity}, we can find a polytope $P$ with $\frac{1}{\fineCD{P}} = \frac{k}{3}$ which has this Fine core normal configuration by scaling $\Delta_2$ by any $k \in \Z_{>0}$, as desired.
\end{proof}

\begin{remark} \label{rem: I2Numerators}
In the notation of Definition \ref{def: IdNumerators}, we obtain $\mathcal{I}_2 = \{2,3\}$.
\end{remark}

\subsection{$d=3$}
\label{sec:d=3}
In order to study higher dimensional cases, we again need a classification
of lattice polytopes with the origin as their only interior lattice
point. This proves already challenging in dimension 3. For this we have used the classification of the toric canonical Fano 3-folds from the Graded Ring Database as developed in \cite{Kas08}. This database contains 674,688 such polytopes. Naively, one could try to determine all values of $\mu^F$ by the Spectrum exploring process mentioned in our discussion following \eqref{eq: milp}, but we will illustrate how to filter out this database instead.

\begin{prop} \label{prop: I3Numerators}
The set of numerators of the rational number $\fineCD{P}$ where $P$ ranges over all $3$-dimensional lattice polytopes is
$$\mathcal{I}_3 \subseteq \mathcal{I}_2 \cup \{ 4,5,7,11,13,17,19,20 \}.$$
\end{prop}

\begin{proof}
By the Graded Ring Database classification we have 674,688 possible
Fine core normal configurations in dimension $3$. Using Lemma
\ref{lemma: dimdcircuits} and by Remark \ref{rem: I2Numerators}, we
know that $\mathcal{I}_2 = \{2,3\} \subset \mathcal{I}_3$ and we are
from this point on, interested only in 4,979 configurations. By
Corollary~\ref{thm: FiniteNumerators}
we need to consider only those
Fine core normal configurations whose vertices define a simplex. After
filtering, we are left with only $8$ such configurations, namely the
ones given by
\begin{equation}\label{eq: 3DimConfigurations}
\begin{split}
\mathcal{A}_{19} & = \{ [1,-2,-5],[-2,2,7],[0,1,0],[1,0,0] \}\\
    \mathcal{A}_{17} & = \{ [-1,1,-5],[1,-2,7],[0,1,0],[1,0,0] \}\\
    \mathcal{A}_{11} & = \{ [1,0,0],[0,1,0],[-2,1,5],[1,-1,-3] \}\\
    \mathcal{A}_{13} & = \{ [-1,1,-4],[1,-2,5],[0,1,0],[1,0,0] \}\\
    \mathcal{A}_{20} & = \{ [-2,2,-5],[1,-3,5],[0,1,0],[1,0,0] \}\\
    \mathcal{A}_{7} & = \{ [-1,1,-2],[1,-2,3],[0,1,0],[1,0,0] \}\\
    \mathcal{A}_5 & = \{ [-1,1,-1],[1,1,2],[0,1,0],[1,0,0] \}\\
    \mathcal{A}_4 & = \{ [-1,-1,-1],[0,0,1],[0,1,0],[1,0,0] \}\\
\end{split}
\end{equation}
whose corresponding numerators are 19, 17, 11, 13, 20, 7, 5
and 4 respectively. By Theorem \ref{thm: FiniteNumerators}, the result follows.
\end{proof}

From our previous discussions and results, we deduce the classification of the Fine spectrum in dimension $3$.

\begin{theorem} \label{thm: Classification3d}
The $3$-dimensional Fine $\Q$-codegree spectrum has the form
\[
S^F(3) = \left\{ \frac{q}{k \cdot \ell} \Bigl\vert q \in
  \{2,3,4,5,7,11,13,17,19\}, \ell \in S_q \text{ and } k \in \Z_{> 0}
\right\}.
\]
where 
$S_q$ is a non-empty set of positive integer numbers depending on $q$. 
\end{theorem}

\begin{proof}
By Proposition \ref{prop: I3Numerators} if the vertices of a $3$-dimensional Fine core normal configuration satisfy that there is a linear relation with positive coefficients between $2$ or $3$ of them, then by Theorem \ref{thm: 2DimSpectrum}, $\frac{1}{\fineCD{P}} = \frac{k}{2}$ or $\frac{1}{\fineCD{P}} = \frac{k}{3}$ for some $k \in \Z_{>0}$. For a given such $k \in \Z_{>0}$, we can find a $3$-dimensional polytope $P$ with $\frac{1}{\fineCD{P}} = \frac{k}{2}$ or $\frac{1}{\fineCD{P}} = \frac{k}{3}$ due to the arguments from \S\ref{sec:project-lift}. This implies $S_2 = S_3 = \{1\}$.

Hence, if $P$ is a $3$-dimensional lattice polytope whose Fine core is zero-dimensional, $\fineCD{P}$ can only have numerators at most in the set
\[
\{4,5,7,11,13,17,19,20\},
\]
obtained from the Fine core normal configurations in \eqref{eq: 3DimConfigurations}. Our classification is then complete by finding polytopes whose Fine $\Q$-codegree have such numerators and then using~\eqref{eq:homgeneity} for determining $S_q$. 

We set $S_4 = \{1\}$ since $\fineCD{\Delta_3} = 4$.

From the computation of the Fine $\Q$-codegree of low-volume polytopes
presented in Table \ref{fig:tablespec}, for $q=5$ we find
\[
P_5 = \operatorname{conv}([0,0,0],[1,0,0],[1,2,0],[1,0,2]), \quad
\fineCD{P_5} = \frac{5}{2}
\]
so we see $2 \in S_5$ but $1 \notin S_5$. Similarly, for $q=7$ we find
\[
P_7 = \operatorname{conv}([0,0,0],[1,0,0],[0,1,0],[3,8,15]), \quad
\fineCD{P_7} = \frac{7}{6}.
\]
Hence $6 \in S_7$. 

For $q \in \{11,13,17,19,20\}$ we construct a polytope via the
procedure mentioned above, and obtain for $q=11$ the lattice polytope
\[
P_{11} = 5 \cdot \operatorname{conv} \left( [10, -1, 4], \left[-2, -1,
    - \frac{4}{5} \right], [-2, -1, 0], [ -2, 5, -2] \right),
\]
so that $\fineCD{P_{11}} = \frac{11}{60}$ and $60 \in S_{11}$. For
$q=13$, we obtain
\[
P_{13} = 15 \cdot \operatorname{conv}\left([15, -2, -4], [-1, -2,
  0],\left[-1, -2, -\frac{4}{5}\right],\left[-1, \frac{10}{3},
    \frac{4}{3}\right]\right),
\]
such that $\fineCD{P_{13}} = \frac{13}{240}$, and thus $240 \in
S_{13}$. Similarly for $q=17$, we obtain
\[
P_{17} = 21 \cdot \operatorname{conv}\left([9, -2, -2],[-1, -2,
  0],\left[-1, -2, -\frac{4}{7}\right],\left[-1, \frac{14}{3},
    \frac{4}{3}\right]\right),
\]
with $\fineCD{P_{17}} = \frac{17}{420}$, so that $420 \in
S_{17}$. Taking $q=19$ we get the polytope
\[
P_{19} = 21 \cdot \operatorname{conv}\left( \left[\frac{37}{3} , -2,
    \frac{11}{3}\right],\left[-1,-2,-\frac{1}{7}\right],\left[-1,-2,-1\right],\left[
    -1,8,-3 \right]\right),
\]
for which $\fineCD{P_{17}} = \frac{19}{840}$. Hence $840 \in
S_{19}$. Finally, in the case $q=20$, we note that the determinant of
$\begin{bmatrix} A & b \end{bmatrix}$ is a multiple of $5$
independently of the choice of $b$. Hence the values for $q=20$ are
obtained already by inspecting the polytopes whose Fine $\Q$-codegree
has numerator $4$, i.e. $20 \notin \mathcal{I}_3$, completing our proof.
\end{proof}

It is worth remarking that in the classification of the $3$-dimensional Fine spectrum, the sets $S_q$ in the statement of Theorem \ref{thm: Classification3d} for $q \geq 5$ are yet to be fully determined.

\subsection{Classification results in dimension $\boldsymbol{d}$}
\label{sec:general-d}

We have so far studied the Fine spectra in lower-dimensional
cases. From the results mentioned above, we can easily deduce that
studying the Fine spectra for dimensions greater than $3$ entails
dealing with a number of Mixed Linear Programs of several orders of
magnitude. Recall that the classification of the spectra is obtained
by solving each of them individually for several fixed $\varepsilon >
0$ and listing all different values. Thus, it is not very feasible to
follow this procedure already in dimension $4$. However, since we are
interested in general classification results for any dimension $d$, we
are able to state some general facts that hold in any dimension. For
this, the classification of degree 1 lattice polytopes from Victor
Batyrev and Benjamin Nill developed in 2006 will be very useful, see
\cite{BN2006Deg1}.

The degree of a $d$-dimensional polytope $P$ satisfies $\deg(P) =
d+1-\cd(P)$. In order to obtain our classification, we need to state
some definitions first.

\begin{definition}
Given lattice polytopes $P_0,...,P_t$ in $\R^k$, the \textit{Cayley
  sum} $P_0 \star \cdots \star P_t$ is the convex hull of
\[
(P_0 \times 0) \cup (P_1 \times e_1) \cup \cdots \cup (P_t \times e_t)
\subseteq \R^k \times \R^t
\]
for $e_1,...,e_t$ the standard basis of $\R^t$.
\end{definition}

\begin{definition}
We call a $d$-dimensional lattice polytope $P$ a \textit{Lawrence
  prism with heights $h_1,...,h_d$} if it is the Cayley Sum of the
segments $[0,h_1], ... , [0,h_d]$.
\end{definition}

\begin{definition}
A $d$-dimensional lattice polytope $P$, for $d \geq 2$, is called
\textit{exceptional} if it is a simplex which is a $(d-2)$-fold
pyramid over $2$ times the $2$-dimensional standard triangle $2
\Delta_2 := \conv \{ (0,0), (2,0), (0,2) \}$.
\end{definition}

We also need the following Proposition.

\begin{prop} \label{prop:exceptionalsimplex}
Let $P$ be a $d$-dimensional exceptional simplex for $d \geq 2$. Then
\[
\fineCD{P} = d - \frac{1}{2}.
\]
\end{prop}

\begin{proof}
The relevant valid inequalities for $P$ are given by $x_1 \geq 0$,
$x_2 \geq 0$ and $-x_1-x_2 \geq -2$. So a point $p = (p_1,p_2) \in
\finecore(P)$ satisfies $p_1 = n^F(P)$, $p_2 = n^F(P)$ and $-p_1-p_2 =
-2 + n^F(P)$. Solving for the Fine number, we obtain $n^F(P) =
\frac{2}{3}$. By Theorem \ref{thm: MuFofPyramid}, since $\fineCD{P} >
1$, the result follows by induction.
\end{proof}

The main classification Theorem for degree 1 lattice polytopes can be
stated as follows.

\begin{theorem}{\cite[Theorem 2.5]{BN2006Deg1}} \label{thm:
    BenjaminDeg1}
Let $P$ be a lattice polytope of arbitrary dimension $d$. Then
$\deg(P) \leq 1$ if and only if $P$ is an exceptional simplex or a
Lawrence prism.
\end{theorem}

Moreover, from \cite[Proposition 1.4]{BN2006Deg1} we obtain that a
$d$-dimensional lattice polytope $P$ has $\deg(P)=0$ if and only if
$P$ is a simplex, i.e., its vertices form an affine lattice basis. And
we also have that for a Lawrence prism $P$ with heights $h_1,...,h_d$
which satisfy $\sum_{i=1}^d h_i \geq 2$, then $\deg(P)=1$.

We can now state our main classification Theorem.

\begin{theorem} \label{thm: general d}
Let $P$ be a lattice polytope of dimension $d \geq 2$. Then:
\begin{itemize}
\item[i)] $\fineCD{P} \cong d+1$ if and only if $P \cong \Delta_d$ up to
  unimodular equivalence. Otherwise:
\item[ii)] $\fineCD{P} = d$ if and only if $P$ projects to the standard
  simplex $\Delta_{d-1}$ in dimension $d-1$. 
  If not:
\item[iii)] $\fineCD{P} = d - \frac{1}{2}$ if and only if $P$ is an
  exceptional simplex. And finally:
\item[iv)] $\fineCD{P} = d-1$ if $P$ projects to $\Delta_{d-2}$ and $P$
  does not satisfy the conditions of any of the previous cases.
\end{itemize}
Moreover, the Fine $\Q$-codegree cannot take any further values in the
range $d-1 \leq \fineCD{P} \leq d+1$ other than those specified above.
\end{theorem}

\begin{proof}
First of all, for a $d$-dimensional lattice polytope $P$ we have
\[
\mu(P) \leq \fineCD{P} \leq \cd (P) \leq d+1.
\]
To see $i)$ we make use of the fact that $\mu(P)$ and $\cd(P)$ equal
$d+1$ if and only if $\deg(P)=0$ and $P \cong \Delta_d$, as proven in
\cite{DiR14}. Hence, we can conclude that $\fineCD{P} = d+1$ if and only
if $P \cong \Delta_d$.

Now, recall that under any projection $\pi$, we have that 
\[
\fineCD{\pi(P)} \leq \fineCD{P}.
\]
To prove $ii)$ let us first assume that $\fineCD{P}=d$, i.e.,
$\deg(P)=1$. By Proposition \ref{prop:exceptionalsimplex}, we know
that since $P$ is not an exceptional simplex, then it must be a
Lawrence prism or equivalently, a Cayley sum of one-dimensional
segments. Hence, $P$ projects to $\Delta_{d-1}$. Conversely, if we now
assume that there is a projection $\pi$ mapping $P$ onto
$\Delta_{d-1}$, by our remark above, we have that $d \leq
\fineCD{P}$. But since $\cd(P) = d$, then it must be the case that
$\fineCD{P}=d$, which proves $ii)$.

In order to prove $iii)$, first assume that $\fineCD{P} = d -
\frac{1}{2}$. As in the previous case, we have that $\deg(P)=1$, so by
Theorem \ref{thm: BenjaminDeg1}, since we know $P$ is not a Lawrence
prism, it must be an exceptional simplex. To show the converse, we use
Proposition \ref{prop:exceptionalsimplex} to obtain that $\fineCD{P} = d
- \frac{1}{2}$, as desired.

Finally, to show $iv)$, if $P$ projects to $\Delta_{d-2}$ we know that
$d-1 = \fineCD{\Delta_{d-2}} \leq \fineCD{P}$. But since we are not in any
of the previous cases, by Theorem \ref{thm: BenjaminDeg1}, we know the
classification of degree 1 lattice polytopes is complete, hence it
must be the case that $\deg(P) \geq 2$, which implies $\cd(P) \leq
d-1$ and this forces $\fineCD{P}=d-1$.

To prove there are no further values of $\fineCD{P}$ between $d-1$ and
$d+1$, note that if this were the case, then $\fineCD{P} > d-1$ forces
$\cd(P) > d-1$ or equivalently $\deg(P) < 2$. Given that we have
completely classified the cases where $\deg(P)$ is either $0$ or $1$,
there are no further values the Fine $\Q$-codegree can take in this
interval.
\end{proof}

Therefore, using this Theorem, we are able to determine that the Fine
$\Q$-codegree spectrum for $P$ a $d$-dimensional polytope has the
following form:
\[
\fineSpec{d} = \left\{ d+1, d , d-\frac{1}{2}, d-1, ...  \right\}
\]
where the gap from $d-1$ down to the following value in the spectrum
is still yet to be determined.

\begin{question}
  It would be interesting to determine the end of the next gap:
  \[
    \max \left( \fineSpec{d} \cap \Q_{<d-1} \right) \,.
  \]
\end{question}

\begin{question}
  Our theorem implies that the overall Fine spectrum
  \[
    \mathcal{S}^F \coloneq \bigcup_{d \ge 1} \left( \fineSpec{d} - d - 1
    \right) \ \subset \Q_{\le 0}
  \]
  starts off with a discrete part $0, -1, -3/2, -2, \ldots$. Is it
  true that it stays discrete beyond $-2$? Where is the first
  accumulation point? At what point does it become dense?
\end{question}

\providecommand{\bysame}{\leavevmode\hbox to3em{\hrulefill}\thinspace}
\providecommand{\href}[2]{#2}

\bibliography{bibliography}
\bibliographystyle{alpha}

\end{document}